\documentclass[12pt]{amsart}
\usepackage{amsfonts,amssymb}
\usepackage{amscd}
\usepackage[colorlinks=true]{hyperref}
\usepackage{hyperref}
\usepackage[usenames]{color}
\hypersetup{colorlinks,citecolor=blue}
\newcommand{\act}{\curvearrowright}

\newcommand{\bmat}{\left(\begin{matrix}}
\newcommand{\emat}{\end{matrix}\right)}

\usepackage{hyperref}
\hypersetup{colorlinks,citecolor=blue}

\newcommand{\BZ}{{\mathbb{Z}}}

\newcommand{\BR}{{\mathbb{R}}}
\newcommand{\BC}{{\mathbb{C}}}

\newcommand{\BP}{{\mathbb{P}}}

\newcommand{\BG}{{\mathbb{G}}}

\newcommand{\gD}{\Delta}

\newcommand{\gC}{\Gamma}

\newcommand{\gs}{\sigma}

\newcommand{\gO}{\Omega}

\newcommand{\gep}{\epsilon}
\newcommand{\gl}{\lambda}

\newcommand{\aut}{\text{Aut}}

\newcommand{\Hom}{\text{Hom}}

\newcommand{\Ad}{\text{Ad}}

\newcommand{\diag}{\text{diag}}

\newcommand{\SL}{\text{SL}}
\newcommand{\GL}{\text{GL}}

\newcommand{\SU}{\text{SU}}

\newtheorem{prop}{Proposition}[section]
\newtheorem{thm}[prop]{Theorem}
\newtheorem{lem}[prop]{Lemma}
\newtheorem{cor}[prop]{Corollary}
\newtheorem{conj}[prop]{Conjecture}
\theoremstyle{definition}

\newtheorem{defn}[prop]{Definition}
\newtheorem{rem}[prop]{Remark}

\newtheorem{exam}[prop]{Example}
\newtheorem{prob}[prop]{Problem}
\newtheorem{ques}[prop]{Question}

\newtheorem{clm}[prop]{Claim}

\newcommand\RR{{\mathcal{R}}}

\newcommand\C{\mathbb{C}}

\newenvironment{dedication}
        {\vspace{6ex}\begin{quotation}\begin{center}\begin{em}}
        {\par\end{em}\end{center}\end{quotation}}

\catcode`\@=11
\long\def\@savemarbox#1#2{\global\setbox#1\vtop{\hsize\marginparwidth 
  \@parboxrestore\tiny\raggedright #2}}
\catcode`\@=12

\begin{document}
\author{Tsachik Gelander}


\date{\today}

\title{$\aut(F_n)$ actions on representation spaces}

\maketitle

\begin{dedication}
\vspace*{.5cm}{Dedicated to the memory of Jacques Tits}
\end{dedication}

\begin{abstract}
J. Wiegold conjectured that if $n\ge 3$ and $G$ is a finite simple group, then the action of $\aut(F_n)$ on $\text{Epi}(F_n,G)$ is transitive. In this note we consider analogous questions where $G$ is a compact Lie group, a non-compact simple analytic group or a simple algebraic group.
\end{abstract}

\section{Introduction}
Let $n$ be an integer. We usually suppose that $n\ge 3$. Let $G$ be a group and consider the representation space $\text{Hom}(F_n,G)$ of all homomorphisms from the free group $F_n$ into $G$. The group $\aut(F_n)$ acts on $\Hom(F_n,G)$ by pre-compositions, 
$$
 \gs\cdot \psi (x)=\psi(\gs^{-1}(x))
$$
where $\gs\in \aut(F_n),\psi\in\Hom(F_n,G)$ and $x\in F_n$. This paper is concerned with various aspects of this action in different contexts. 
By choosing a base for $F_n$ we can identify (non-canonically) $\Hom(F_n,G)$ with $G^n$. 

If $G$ carries an additional structure, $\Hom(F_n,G)$ inherits a related structure which is often preserved by the $\aut(F_n)$ action.
Let us give some examples:
\begin{enumerate}
\item If $G$ is finite so is $\Hom(F_n,G)$ and $\aut(F_n)$ acts by permutations.
\item If $G$ is a topological group, $\Hom(F_n,G)$ with the compact open (equivalently, the pointwise convergence) topology is a topological space and $\aut(F_n)$ acts by homeomorphisms.
\item If $G$ is a locally compact group, the Haar measure on $G$ induces a measure class on $\Hom(F_n,G)$ which is $\aut(F_n)$-invariant.
\item Furthermore, if $G$ is unimodular, the Haar measure on $G$ induces an $\aut(F_n)$-invariant measure on $\Hom(F_n,G)$.
\item If $G$ is a linear algebraic group then $\Hom(F_n,G)$ is an affine algebraic variety and $\aut(F_n)$ acts by algebraic automorphisms.
\end{enumerate}

Most of the time, we focus on the invariant subset of epimorphisms, where in the context of topological (or algebraic) groups we mean homomorphisms with dense (or Zariski dense) image,
$$
 \text{Epi}(F_n,G):=\{f\in\Hom(F_n,G):\overline{f(F_n)}=G\}.
$$ 

An epimorphism $f:F_n\to G$ is called {\it redundant} if there is a non-trivial free product decomposition $F_n=A*B$ such that the restriction of $f$ to $A$ is still an epimorphism, i.e. $\overline{f(A)}=G$, where again the bar denotes the closure (or the Zariski closure) and 
should be ignored in the discrete case. Denote
by $\mathcal{R}(F_n,G)\subset \text{Epi}(F_n,G)$ the set of redundant representations.

One main question we are interested in is whether, under certain natural assumptions on $G$, every epimorphism is redundant. 

\begin{ques}\label{Q:redundant}
Let $G$ be either a finite simple group, a connected compact Lie group or a Zariski connected simple algebraic group and let $n\ge 3$. Does 
$$
 \mathcal{R}(F_n,G)=\text{Epi}(F_n,G)?
$$
\end{ques}

Let us denote by $\mathcal{P}_n\subset F_n$ the set of primitive elements. Recall that an element $p\in F_n$ is primitive iff it belongs to a free basis of $F_n$. Note that if $F_n=A*B$ is a free product decomposition and $b\in B$ is primitive, then the coset $Ab$ consists of primitive elements.
In the discrete case (in particular when $G$ is finite) a map $f:F_n\to G$ is redundant iff $f(\mathcal{P}_n)=G$ or equivalently, $f$ is surjective and $\ker(f)\cap\mathcal{P}_n\ne\emptyset$. When $G$ is a topological group (or an algebraic group) the image of the primitive elements under a redundant map is dense (resp. Zariski dense).

\begin{rem}
Let $n\in\BZ$. Since $F_n$ is finitely generated and residually finite, it is Hopfian, i.e. $\text{Epi}(F_n,F_n)=\aut(F_n)$. 
(Indeed, if $f:F_n\to F_n$ is a surjective homomorphism, and $H\le F_n$ is a subgroup of index $m$ then $f^{-1}(H)$ is of index $m$ as well. Since for every $m<\infty$, there are only finitely many subgroups of index $m$, all these subgroups must be of the form $f^{-1}(H),~[F_n:H]=m$ i.e. they all contain $\text{ker}(f)$, and since $F_n$ is residually finite, it follows that $\text{ker}(f)=\{1\}$.)
Moreover, it follows from the Nielsen theorem that for every $k\le n$, $\aut(F_n)$ acts transitively on $\text{Epi}(F_n,F_k)$. In particular for $k<n$ any epimorphism $F_n\to F_k$ is redundant. See \cite[Chapter 2]{LS} for a detailed description of the Nielsen theorem (in particular see \cite[Proposition 2.12]{LS}).
It follows that for $G$ discrete, an epimorphism $f:F_n\to G$ is redundant iff it factors through an epimorphism from $F_k$ to $G$ with $k<n$.
Nielsen's theorem also produces a convenient generating set for $\aut(F_n)$, namely the Nielsen transformations. With respect to an ordered free basis $(x_1,\ldots,x_n)$ for $F_n$ the Nielsen transformations take the form:
\[
 R_{i,j}^{\pm}:(x_1,\ldots,x_j,\ldots,x_n)\mapsto (x_1,\ldots,x_jx_i^{\pm 1},\ldots,x_n),~i\ne j
\]
\[
 L_{i,j}^{\pm}:(x_1,\ldots,x_j,\ldots,x_n)\mapsto (x_1,\ldots,x_i^{\pm 1}x_j,\ldots,x_n),~i\ne j
\]
\[
 I_j:(x_1,\ldots,x_j,\ldots,x_n)\mapsto (x_1,\ldots,x_j^{-1},\ldots,x_n).
 \]
\end{rem}


\subsection{Finite simple groups}

The following famous conjecture is attributed to J. Wiegold. 

\begin{conj}[Wiegold]
For any finite simple group $G$ and $n\ge 3$, the action of $\aut(F_n)$ on $\text{Epi}(F_n,G)$ is transitive.
\end{conj}

It was shown in \cite{Gi,Ev93b} that all the redundant representations are in the same $\aut(F_n)$-orbit. Thus the Weigold conjecture is equivalent to the conjecture that any epimorphism from $F_n$ to $G$ is redundant, i.e. to an affirmative answer to Question \ref{Q:redundant}
in this context.

We will not consider here finite simple groups. For what is known and mostly not known concerning the Wiegold conjecture, as well as interesting applications and related questions we refer the reader to the wonderful exposition by Lubotzky \cite{Lub}.\footnote{There is a small overlap between this paper and Lubotzky's exposition \cite{Lub}, but they are mostly complementary.}
Instead, we will consider analogous problems in the context of compact Lie groups, non-compact semisimple Lie groups, semisimple groups over non-archimedean local fields and mainly semisimple algebraic groups.

\begin{rem}
The assumption $n\ge 3$ is necessary (here and below). Indeed if $G$ is a group of matrices (over any field) then the function
$$
 \text{Trace} [f(x_1),f(x_2)]+\text{Trace} [f(x_2),f(x_1)]
$$
where $x_1,x_2\in F_2$ are free generators, and $[\cdot,\cdot]$ denotes the commutator, is $\text{Aut}(F_2)$ invariant and in general non-constant.
\end{rem}


\section{Compact Lie groups}\label{section:compact}

In this section $G$ denotes a connected compact Lie group. We consider the compact open topology on $\Hom(F_n,G)$. Since $G$ is compact, this is a compact space.
The normalized Haar measure $\mu$ on $G$ induces a probability measure $\mu_n$ on $\Hom(F_n,G)$; concretely, by fixing a base for $F_n$ we may identify $\Hom(F_n,G)$ with $G^n$ and take the product measure $\mu_n=\mu^n$. Since $G$ is unimodular, $\mu_n$ is $\aut(F_n)$-invariant. Indeed, it is enough to verify this claim for Nielsen transformations, whose effect on $G^n$ can be described by multiplying one coordinate from the left or from the right by another or by taking the inverse of one coordinate, and the Haar measure is invariant under these operations.

What makes the compact Lie group case simpler than other scenarios is that up to measure zero we have:
$$
 \mathcal{R}(F_n,G)=\text{Epi}(F_n,G)=\Hom(F_n,G),
$$
where the second equality holds for $n\ge 2$ and the first one for $n\ge 3$. 
Furthermore, if $G$ is perfect, i.e. $G=G'$, then $\mathcal{R}(F_n,G)$ is also open.
These facts are immediate consequences of the following lemma (see \cite{Ge08} for more details):

\begin{lem}\label{lem:pairs}\cite[Lemma 1.4]{Ge08}
Let $G$ be a connected compact Lie group.
Let 
$$
 \gO=\{(x,y)\in G^2: \langle x,y\rangle~\text{is dense in}~G\}.
$$ 
Then $\gO$ is dense and of full measure. Furthermore, if $G$ is perfect, it is also open.
\end{lem}

The ergodicity part of the following theorem from \cite{Ge08} confirmed a conjecture of Goldman, and the minimality, though not stated in \cite{Ge08}, is implicit in the  
argument which served to prove of a conjecture of Margulis and Soifer. Recall that a group action on a topological space is minimal if every orbit is dense.

\begin{thm}\label{thm:Goldman-Conjecture}
Let $G$ be a connected compact semisimple Lie group and let $n\ge 3$. Then the $\aut(F_n)$ action on $\mathcal{R}(F_n,G)$ is both ergodic and minimal. 
\end{thm}

The theorem is formulated using new terminology and presenting its short proof will allow us to introduce additional useful terms and ideas.

Fix a basis $\{x_1,\ldots,x_n\}$ of $F_n$ and set 
$$
 \mathcal{O}_i=\{f:F_n\to G: \langle \{ f(x_j),j\ne i\}\rangle~\text{is dense in}~G\},~i\in\{1,\ldots,n\}.
$$
Then $\mathcal{R}(F_n,G)=\bigcup_{\gs\in\aut(F_n)}\gs\cdot\mathcal{O}_n$. Let us also define the set of {\it strongly redundant} representations as follows:
$$
 \mathcal{SR}(F_n,G)=\bigcap_{\gs\in\aut(F_n)}\gs\cdot\mathcal{O}_n.
$$
Then $\rho\in\Hom(F_n,G)$ is strongly redundant iff its restriction to any free factor of rank $n-1$ is dense.
Since $\aut(F_n)$ is countable, it follows from Lemma \ref{lem:pairs} that $\mathcal{SR}(F_n,G)$ is of full measure and of second category.

\begin{proof}[Proof of ergodicity]
In order to prove that the action of $\aut(F_n)$ on $\text{Hom}(F_n,G)$ is ergodic, it is enough to restrict to the invariant subset $\mathcal{SR}(F_n,G)$. Suppose by way of contradiction that there is an $\aut(F_n)$ invariant subset $A\subset \mathcal{SR}(F_n,G)$ which is neither null nor co-null. Making use again of the fact that $\aut(F_n)$ is countable, we may suppose that $A=\bigcap_{\aut{F_n}}\gs\cdot A$, i.e. that $A$ is set-theoretically invariant rather than measurably. 
In view of Fubini's theorem the set $A$ is not independent of at least one coordinate, i.e. 
there must be $i\in\{1,\ldots,n\}$ and $g_j\in G, j\ne i$ such that 
$$
 B=\{f(x_i):f\in A, f(x_j)=g_j, \forall j\ne i\}
$$
is neither null nor co-null in $G$. Since $A$ is $\aut(F_n)$-invariant, $B$ is invariant under left multiplication by $g_j,~j\ne i$. But $\langle \{ g_j, j\ne i\}\rangle$ is dense and hence acts ergodically on $G$. A contradiction. 
\end{proof}

For the proof of minimality it is convenient to introduce some more notations. 
For a subset $S\subset G$ let $\mathcal{D}(S)$ denote the set of elements which together with $S$ generate $G$:
$$ 
 \mathcal{D}(S):=\{g\in G: \overline{\langle S,g\rangle}=G\}.
$$

Suppose that $G$ is as in Theorem \ref{thm:Goldman-Conjecture}.

\begin{lem}
If $\mathcal{D}(S)$ is nonempty then it is open and dense.
\end{lem}

\noindent
This is implicit in \cite[\S 1.1]{Ge08}.\qed

\medskip

Let $\mathcal{D}'(S)=\bigcap_{s\in S}\mathcal{D}(S\setminus\{s\})$. 

\begin{lem}
Suppose that $S$ is a finite (topological) generating set for $G$, then $\mathcal{D}'(S)$ is open and dense. 
\end{lem} 

\begin{proof}
For $s\in S$ the set $\mathcal{D}(S\setminus\{s\})$ is nonempty since it contains $s$. Thus $\mathcal{D}'(S)$ is an intersection of finitely many open dense sets.
\end{proof}

Abusing notations, for $f\in \Hom(F_n,G)$ let us denote 
$$
 \mathcal{D}'(f)=\mathcal{D}'(\{f(x_1),\ldots,f(x_n)\}).
$$ 
The above lemma says that for $f\in \text{Epi}(F_n,G)$, $\mathcal{D}'(f)$ is open and dense in $G$. 
Let $\gO=\cap_{i=1}^n\mathcal{O}_i$.
That is, 
$$
 \gO=\{f\in\Hom(F_n,G):\overline{\langle \{ f(x_j),j\ne i\}\rangle}=G,~\forall i\in\{1,\ldots,n\}\}.
$$

\begin{proof}[Proof of minimality]
Let $f\in \mathcal{R}(F_n,G)$ be a redundant representation. Up to applying some element from $\aut(F_n)$ we may suppose that $f\in \mathcal{O}_n$. Then $f(x_1),\ldots,f(x_{n-1})$ generate a dense subgroup and hence by acting by left multiplications on the last coordinate we can take it into $\mathcal{D}'(\{f(x_1),\ldots,f(x_{n-1})\})$. This allows us to assume that $\gs\cdot f\in \gO$ for an appropriate $\gs\in\aut(F_n)$.
Now $\gO$ is open and dense in $\Hom(F_n,G)$ and by applying transformations 
corresponding to left multiplications of one coordinate by words in the other coordinates 
we may move every coordinate into any neighborhood in $G$ without leaving $\gO$. This shows that the orbit of $f$ is dense in $\Hom(F_n,G)$.
\end{proof}

Since the ergodicity result holds for any compact connected group $G$\footnote{Even the assumption that $G$ is semisimple is unnecessary (see \cite{Ge08}).}, 
we obtain for free:

\begin{cor}
Let $n\ge 3$ and $G$ a connected compact Lie group. Then the action of $\aut(F_n)$ on $\Hom(F_n,G)$ is weakly mixing.
\end{cor}

Recall that a p.m.p action $\gC\act X$ is weakly mixing if and only if it is doubly ergodic, i.e. if the diagonal action $\gC\act X\times X$ is ergodic.

\begin{proof}
We need to show that the action is doubly ergodic. However 
$$
 \Hom(F_n,G)\times\Hom(F_n,G)
$$ 
is canonically isomorphic to $\Hom(F_n,G\times G)$. Since $G^2$ is a connected compact Lie group, the result follows from Theorem \ref{thm:Goldman-Conjecture}.
\end{proof}


\subsection{Invariant measure classification}
Given a closed subgroup $H\le G$, the Haar measure of $H$ induces an invariant measure on the invariant set $\text{Epi}(F_n,H)\subset\Hom(F_n,G)$. We shall call such a measure {algebraic}. More generally we say that an invariant measure $\mu$ on $\Hom(F_n,G)$ is {\it algebraic} if there is a closed subgroup $H\le G$ such that $\mu$ is supported on $\text{Epi}(F_n,H)$ and the restriction of $\mu$ to each fiber of the quotient $\text{Epi}(F_n,H)\to\text{Epi}(F_n,H/H^\circ)$ coincides with the Haar measure of $(H^\circ)^n$, where $H^\circ$ is the identity connected component of $H$. This formulation allows us to avoid dealing with finite groups. In analogy to Ratner's theorem we conjecture:

\begin{conj}[Measure classification conjecture, suggested by Nir Avni]
Let $G$ be a compact Lie group and $n\ge 3$. Then every $\aut(F_n)$-invariant ergodic measure on $\Hom(F_n,G)$ is algebraic.
\end{conj}

Restricting to connected groups we have the following formulation:

\begin{conj}[Unique ergodicity conjecture]
Let $G$ be a connected compact Lie group and $n\ge 3$. Then the only $\aut(F_n)$-invariant ergodic measure on $\text{Epi}(F_n,G)$ is the Haar measure on $G^n$.
\end{conj}

A weaker conjecture, in the spirit of the `baby Wiegold conjecture' discussed in \S \ref{sec:BB} is that the character variety admits no invariant atomic measures:

\begin{conj}
Let $G$ ba a compact connected semisimple Lie group, and lat $n\ge 3$. Then $\chi_n(G)=\text{Hom}(F_n,G)/G$ admits no atomic $\aut(F_n)$-invariant measures.
\end{conj} 

Here $\chi_n(G)=\text{Hom}(F_n,G)/G$ is the space quotient of $G$-orbits where $G$ acts on $\text{Hom}(F_n,G)$ by conjugation $(g\cdot f)(x)=gf(x)g^{-1}$.

\medskip

As we have seen almost every representation is redundant (w.r.t the Haar measure). A much more delicate question is whether {\bf every} dense representation is redundant:

\begin{ques}
Given $n\ge 3$ and a compact connected Lie group $G$, does 
$\mathcal{R}(F_n,G)=\text{Epi}(F_n,G)$? Where the equality here is set theoretically and not up to measure $0$. 
\end{ques}



\section{Non-compact semisimple groups}

The non-compact semisimple case is more involved even when restricting to the Haar measure. 
In particular the answer to Question \ref{Q:redundant} is negative in certain cases. Yair Minsky showed in \cite{Mi} that for $G=\SL(2,\BR)$ or $\SL(2,\BC)$, the set $\text{Epi}(F_n,G)\setminus\mathcal{R}(F_n,G)$ has a non-empty interior. 
On the other hand Yair Glasner showed in \cite{Gl} that for $G=\SL(2,k)$ with $k$ a non-archimedean local field, the set $\text{Epi}(F_n,G)\setminus\mathcal{R}(F_n,G)$ has (Haar) measure zero.

Y. Minsky considered the action of $\text{Out}(F_n)$ on 
$\Hom(F_n,G)/ G$.\footnote{In fact Minsky considered the action on the space of characters $\chi(F_n,G)=\Hom(F_n,G)// G$
which is defined in the sense of geometric invariant theory and has the structure of an algebraic variety. However that space differs from the purely topological quotient $\Hom(F_n,G)/ G$ only at reducible points, and hence the two agree when restricting to representations with dense image (see \cite{Mi} and the references therein).} 
He considered $\text{Out}(F_n)$ rather than $\aut(F_n)$ since the inner automorphisms acts trivially on that quotient.
He defined the notion of primitive-stable representations and proved that the set $\mathcal{PS}(F_n,G)$ of primitive stable representations is open and the action of $\text{Out}(F_n)$ on its image in $\Hom(F_n,G)/G$ is properly discontinuous and in particular non-ergodic. Furthermore, he proved that $\mathcal{PS}(F_n,G)$ contains a dense representation, i.e. the open invariant set $\mathcal{PS}(F_n,G)\cap \text{Epi}(F_n,G)$ is non-empty. It follows directly from the definition that the image of the primitive elements $\mathcal{P}_n$ under a primitive stable representation misses a neighborhood of the identity of $G$. In particular one deduces:

\begin{thm}[Minsky]
Let $G=\SL(2,\BC)$ or $\SL(2,\BR)$. Then there is a dense representation $\rho:F_n\to G$ such that the image of the primitive elements $\rho(\mathcal{P}_n)$ is not dense in $G$.
\end{thm}

In contrast to Minsky's result, Y. Glasner \cite{Gl} proved that for $G=\SL(2,k)$ with $k$ non-Archimedean and $n\ge 3$ the action of $\aut(F_n)$ on $\text{Epi}(F_n,G)$ is ergodic. The dichotomy between the Archimedean and the non-Archimedean cases for $\SL_2$ can be named the Yairs' dichotomy. For more details here about these results, see the well written papers by Y. Minsky and Y. Glasner. \cite{Mi,Gl}.  


Note that the ergodicity question brakes into two independent questions:
\begin{itemize}
 \item Is the $\aut(F_n)$ action on the redundant part $\mathcal{R}(F_n,G)$ ergodic?
 \item Does $\text{Epi}(F_n,G)=\mathcal{R}(F_n,G)$ up to measure zero?
\end{itemize}
Jointly with Minsky we answered the first question and also proved a minimality results in a quite general context, answering a conjecture of Lubotzky:

\begin{thm}[\cite{GM}]
Let $k$ be a local field, $\BG$ a Zariski connected simple
$k$ group, and $G=\BG(k)$ the group of $k$ points. If $\text{char}(k)>0$ assume further that the adjoint representation of $G$ is irreducible.
Then the action of $\aut(F_n)$ on $\mathcal{R}_n(G)$ is ergodic with respect to the Haar measure and minimal with respect to the compact open topology.
\end{thm}

On the other hand, unlike the compact case, we showed that this action is not always weak mixing in the following sense:

\begin{thm}[\cite{GM}]
Let $G=\SL_2(\BR)$ or $\SL_2(\BC)$, and $n\ge 3$. Then the action of
$\text{Out}(F_n)$ on $\overline\RR_n(G):=\RR(F_n,G)/G$ is not weakly mixing,
in the sense that the diagonal action of $Out(F_n)$ 
on $\overline\RR_n(G)\times\overline\RR_n(G)$ is not ergodic. 
\end{thm}

For the $\text{Out}(F_n)$ action on $\text{Hom}(F_n,G)/G$ Lubotzky suggested a dynamical decomposition conjecture, namely that the space decomposes up to measure zero to the redundant part, on which the action is ergodic, and the primitive stable part, on which the action is properly discontinuous. 

\begin{conj}[Lubotzky]
$\text{Hom}(F_n,G)/G=\mathcal{R}(F_n,G)\cup\mathcal{PS}(F_n,G)$ up to Haar measure zero.
\end{conj}

This conjecture is still open even for $\SL(2,\BR)$.


\section{Linear algebraic groups}
 A beautiful conjecture of Zelmanov \cite{Zel} states that if $\rho:F_n\to \GL(d,\BC)$ is a representation such that all primitive $p\in F_n$ are mapped to elements sharing the same characteristic polynomial, then $\rho(F_n)$ is virtually solvable.
Let us formulate a more general conjecture in the spirit of Question \ref{Q:redundant}.
Let $G$ be a simple complex linear algebraic group. We suppose below that $n\ge 3$.

\begin{defn}
We say that a representation $f:F_n\to G$ is {\it Zariski redundant} if for some proper free factor $A<F_n$, $f(A)$ is Zariski dense. We let $\mathcal{R}_n(G)$ denote the set of Zariski redundant representations.
\end{defn}

\begin{conj}\label{conj:0}
A Zariski dense representation $f:F_n\to G$ is Zariski redundant. That is $\text{Epi}(F_n,G)=\mathcal{R}_n(G)$.
\end{conj}

Let us give two alternative formulations for this conjecture.
Note that $\text{Hom}(F_n,G)\cong G^n$ inherits a structure of an affine algebraic variety.

\begin{conj}\label{conj}\label{conj:2}
Let $\rho\in\text{Hom}(F_n,G)$ be a Zariski dense representation. Then the orbit $\aut(F_n)\cdot\rho$ is Zariski dense in the representation variety $\text{Hom}(F_n,G)$.
\end{conj}


The equivalence between Conjecture \ref{conj:2} and Conjecture \ref{conj:0} can be deduced from:

\begin{lem}\label{lem:Z-dense-open}
Let $G$ be a simple complex algebraic group.
The set 
$$
 \{(a,b)\in G^2:\langle a,b\rangle~\text{is Zariski dense}\}
$$ 
contains a Zariski open subset of $G^2$.
\end{lem}

\begin{proof}
Consider the associative algebra $\mathcal{A}=\text{span}(\Ad(G))$.
By a theorem of Jordan \cite{J}, there is $M=M(G)$ such that every finite subgroup of $G$ contains a normal abelian subgroup of index at most $M$. Set $m=M!$. Pick
$x,y\in G$ such that $[x^m,y^m]\ne 1$ and $\langle x,y\rangle$ is Zariski dense. Let $W_1,\ldots W_d$ be words in $F_2$ such that $\{ W_i(x,y)\}_1^d$ spans $\mathcal{A}$. Set 
$$
 U=\{ (a,b): \text{span}\{ W_i(a,b)\}_1^d=\mathcal{A},~\text{and}~[a^m,b^m]\ne 1\}
$$ 
then $U$ is nonempty and Zariski open. Now for $(a,b)\in U$, the group $\langle a,b\rangle$ must be infinite, i.e. the connected component $H$ of its Zariski closure is nontrivial. 
However $\text{Lie}(H)$ is $\mathcal{A}$ invariant and hence $H$ is normal. Therefore $H=G$.
\end{proof} 

\begin{cor}
The set $\mathcal{R}_n(G)$ contains a Zariski open subset.
\end{cor}

\begin{proof}
The desired Zariski open set can be defined as 
$$
 \{f\in\text{Hom}(F_n,G): (f(x_1),f(x_2))\in U\}
$$ 
where $U$ is the set defined in the previous proof and $\{ x_1,\ldots,,x_n\}$ is a fixed basis for $F_n$.
\end{proof}

In view of this Corollary the implication \ref{conj:2} $\Rightarrow$ \ref{conj:0} is immediate. The converse implication \ref{conj:0} $\Rightarrow$ \ref{conj:2} 
can be proved by virtually the same argument that shows the minimality in \S \ref{section:compact}. The following conjecture is a priory weaker:

\begin{conj}\label{conj:1}
If $\rho(F_n)$ is Zariski dense in $G$ then so is $\rho(\mathcal{P}_n)$.
\end{conj}

For instance to see that  \ref{conj:2} $\Rightarrow$ \ref{conj:1} one may identify $\text{Hom}(F_n,G)$ with $G^n$ and consider the projection to the first factor.

Let us explain also why a representation $\rho$ which satisfies \ref{conj:1} also satisfies \ref{conj:0}, at least under some mild assumptions. Suppose for simplicity that $n=3$. Identify $\Hom(F_3,G)$ with $G^3$ by choosing a base for $F_3$. Assuming \ref{conj:1} we may find a point $(a,b,c)$ in the orbit such that $a$ is regular, i.e. $C_G(a)$ is a maximal torus. There are only finitely many parabolic subgroups $Q_i$ containing $a$ and for every $x\notin\cup Q_i$ the group $\langle a,x\rangle$ is Zariski dense. Now if for some $k\in\BZ$, $b^kc\notin\cup Q_i$ then we are done. Otherwise, for infinitely many $k$'s $b^kc\in 
Q_{i_0}$ implying that $\overline{\langle b\rangle}^zc\subset Q_{i_0}$. In the case where $ \overline{\langle b\rangle}^z$ is Zariski connected, this gives that $a,b,c\in Q_{i_0}$, a contradiction.



\section{Linear representations of $\aut(F_n)$}

A celebrated theorem of Formanek and Procesi \cite{FP} says that for $n\ge 3$ the group $\aut(F_n)$ is not linear. The proof of that theorem inspired several people to conjecture the following: 

\begin{conj}[Formanek--Zelmanov, Lubotzky]\label{conj:FZL}
Let $n\ge 3$ and let $F\lhd\aut(F_n)$ be the subgroup of inner automorphisms. Then the image $\rho(F)$ of $F$ under any linear representation $\rho:\aut(F_n)\to\GL_d(\BC)$ is virtually solvable.
\end{conj}

We remark that the conjectures presented in the previous section imply Conjecture \ref{conj:FZL}. Indeed, if $\rho$ is a linear representation of $\aut(F_n)$, the restriction of $\rho$ to $F_n=\text{Inn}(F_n)$ is a representation for which the image of the set of primitive elements $\rho(\mathcal{P}_n)$ is contained in a single conjugacy class and hence shares the same characteristic polynomial. 

By a theorem of Platonov and Potapchik \cite{PP}, for any complex linear representation $\rho$ of $\aut(F_n)$, the image of the primitives $\rho(\mathcal{P}_n)$ consists of virtually unipotent elements, i.e., matrices for which all of the eigenvalues are roots of unity. Inspired by that, it is natural to focus on the study of linear representations of $F_n$ for which the primitive elements are sent to either torsion or unipotent elements. In the proceeding sections we will pay special attention to the unipotent case, but first we will elaborate a bit about the torsion case.


\section{Burnside's type problem}

It is well known that a finitely generated torsion linear group must be finite. In other words, the original Burnside's problem has a positive solution for linear groups. What if we require only the primitive elements to be torsion, i.e., that $\gC=\rho(F_n)$ where $\rho(\mathcal{P}_n)$ consists of torsion elements. Does this condition imply that $\gC$ is finite? The following construction shows that it does not.\footnote{A very similar construction, explaining the same phenomenon is given also in \cite{Zel}.} 

\begin{exam}[An infinite linear group with torsion primitives]\label{exam:Burnside}
Let $n>1$. Let $H$ be a finite group generated by $n$ elements but not by less. Embed $H$ in $\text{SO}(d)$ for an appropriate $d$ such that no nontrivial element in $H$ admits a fixed point in $S^{d-1}$. Pick $n$ generators $h_i$ for $H$ and define a group $\gC\le\text{Isom}(\BR^d)\cong\text{O}(d)\ltimes\BR^d$ by choosing $n$ different points $p_i$ in $\BR^d$
and let $\overline h_i$ be the element rotating like $h_i$ around $p_i$. 
Note that $p_i$ is the only fixed point of $\overline h_i$.
Let $\gC=\langle \overline h_i\rangle$, then $\gC$ has no common fixed point and hence it is infinite. Consider the map $f:F_n\to\gC$ defined by chosing a basis $\{x_1,\ldots,x_n\}$ to $F_n$
and setting $f(x_i)= \overline h_i,~i=1,\ldots,n$.
Since $H$ cannot be generated by less than $n$ elements, for any primitive $x\in\mathcal{P}_n$, $f(x)\in\gC$ projects nontrivially to $\text{SO}(d)$ and hence has no fixed point at $S^{d-1}$. It follows that $f(x)$ fixes a point in $\BR^d$ and hence is of finite order.

Concretely, for $n=2$ we can take a non-commutative finite subgroup $H\le \SU(2)$ and embed $\SU(2)\to\text{SO}(4)$ such that every nontrivial element acts freely. Similarly for $n=3,4$ we can start with $H$ as above and take $H\times H^g\le\SU(2)\times\SU(2)$ for some $g\in\SU(2)$ which does not commute with any $h\in H\setminus\{ 1\}$ and embed $\SU(2)^2$ in $\text{SO}(4)$ by considering the obvious isometry $\SU(2)\cong {S}^3$ and the action $\SU(2)\times\SU(2)\act \SU(2)$ by left and write multiplications $(g_1,g_2)\cdot x:=g_1xg_2^{-1}$.

\end{exam}

However, the group $\gC$ in Example \ref{exam:Burnside} is virtually abelian. Indeed the kernel of the projection of $\gC$ to the orthogonal group $\text{O}(d)$ is of index
$|H|$ and consists of translations.
This example leads to the following, more educated, formulation. We suppose again that $n\ge 3$: 

\begin{conj}
Let $f:F_n\to\GL_d(\BC)$ be a representation such that $f(\mathcal{P}_n)$ consists of torsion elements, then $f(F_n)$ is virtually solvable. 
\end{conj}

Let us give another variant of this conjecture:

\begin{conj}\label{conj:Zalmanov}
Let $f:F_n\to \SL(d,\BC)$ be a representation for which the Zariski connected component of the image is simple, then $f(\mathcal{P}_n)$ contains an element of infinite order.
\end{conj}

\begin{rem}
As we already explained,
the special case where all the primitive elements belong to the same conjugacy class is of significant interest for applications.
\end{rem}

\begin{prop}\label{prop:4.5}
Conjecture \ref{conj:2} implies Conjecture \ref{conj:Zalmanov}.
\end{prop}

\begin{proof}[Sketched proof of \ref{prop:4.5}]
Suppose $n=3$, the general case is similar.
To simplify the argument we will carry out the argument under the assumption that the image of $f$ is Zariski connected. The same argument can be carried out also without this assumption using techniques developed in \cite{toti,MS}.
Assuming Conjecture \ref{conj:2} (equivalently Conjecture \ref{conj:0}), we have a primitive triple $(a,b,c)$ in $f(F_3)$ such that $\langle a,b\rangle$ is Zariski dense. Let $K$ be a the field generated by the entries of the matrices in the group $\langle a,b,c\rangle$. Since $\langle a,b\rangle$ is infinite there is a local field $k$, which is a completion of $K$ with respect to some valuation, such that $\langle a,b\rangle$ is unbounded in $\SL(d,k)$. 

Since $G$ is Zariski connected and simple (in particular reductive) and $\langle a,b\rangle$ is Zariski dense in $G$, we can find a finite set $F\subset \langle a,b\rangle$ such that
\begin{itemize}
\item for any exterior power $W$ of $\C^d$
\item for any $G$ irreducible component  $V$ of $W$
\item for every projective point $v$ and projective hyperplane $H$ in $\BP(V)$
\end{itemize}
there is $g\in F$ s.t. $g\cdot v\notin H$. 
Furthermore, since finite dimensional Grassmaninan manifolds over $k$ are compact, by choosing metrics (there is a natural choice but it's not important), 
there is $r>0$ such that $F$ is $r$-separating, i.e. for every $v,H$ as above there is $g\in F$  with $d(g\cdot v,H)>r$.
Since $\langle a,b\rangle c$ is unbounded, and we consider only finitely many representations (the irreducible components of the exterior power representations), we can fix one such representation $V$ such that for every $\gep>0$ 
there is an element $c'\in \langle a,b\rangle c$ which is $\gep$-contracting in $\BP(V)$. That is, there is an hyperplane $H=H(c')\subset\BP(V)$ and a projective point $v=v(c')\in\BP(V)$ such that $c'$ takes the complement of the $\gep$-neighborhood of $H$ into the $\gep$-ball around $v$.
Let $L$ be an upper bound for the Lipschitz constants of $F$, and take $\gep<\frac{r}{L+1}$. Then, for a suitable $g\in F$, we have $d(g\cdot v,H)>r$. But then $c''=gc'$ takes the complement of the $\gep$-neighborhood of $H=H(c')$ into the $L\gep$-ball around $g\cdot v$. Since these two sets are disjoint, the element $c''$ is of infinite order.
Recall that $c''$ is primitive since it belongs to $\langle a,b\rangle c$.
\end{proof}


\section{Primitive-unipotent representations}

Let $U(d)\subset \GL(d,\BC)$ denote the set of $d\times d$ unipotent complex matrices. That is $U(d)$ is the set of matrices whose characteristic polynomial is $(x-1)^d$.
Note that the logarithm restricted to $U(d)$ is a polynomial isomorphism between $U(d)$ and the set of nilpotent matrices in $M_d(\BC)$, and its inverse $\exp$ is also a polynomial map. 
This implies that the cyclic group generated by any non-trivial element $u\in U(d)$ is Zariski connected and one dimensional, concretely, $\overline{\langle u\rangle}^Z=\exp(\{t\log u:t\in\BC\})$.

We say that a representation $\rho:F_n\to\GL(d,\BC)$ is a {\it unipotent representation} if $\rho(F_n)\subset U(d)$. 
If $\rho$ is a unipotent representation then $\rho(F_n)$ is a Zariski connected nilpotent group which can be conjugated into the group of upper triangular unipotents. 
This follows from the fact that an algebraic group which is not nilpotent contains a non-trivial semisimple element which is a consequence of the Levi decomposition. In fact if $\rho$ is a unipotent representation then the Lie algebra of the Zariski closure of $\rho(F_n)$ is the vector space $\text{span}(\log(\rho(F_n)))$ which consists of nilpotent matrices and form a Nilpotent Lie algebra by Engel's theorem.

We denote by $\mathcal{U}(n,d)$ the variety of $d$-dimensional unipotent representations of $F_n$. 
We shall say that a representation $\rho:F_n\to\GL(d,\BC)$ is {\it primitive-unipotent} if the image of any primitive element is unipotent, that is if $\rho(\mathcal{P}_n)\subset U(d)$. We shall denote by $\mathcal{PU}(n,d)$ the variety of $d$-dimensional primitive-unipotent representations of $F_n$. Obviously, $\mathcal{U}(n,d)\subset\mathcal{PU}(n,d)$. Platonov and Potapchik asked in \cite{PP} whether every primitive-unipotent representation is unipotent, i.e. whether the two varieties coincide. 
We conjecture that for $n\ge 3$ the answer to this questions is positive:

\begin{conj}[Platonov-Potapchik]\label{conj:PP}
The two varieties $\mathcal{U}(n,d)$ and $\mathcal{PU}(n,d)$ coincide, for every $n\ge 3$ and $d$.
\end{conj}  

We suppose below that $n\ge 3$.
It follows from the last line of the first paragraph of this section that if $\rho\in \mathcal{PU}(n,d)$ then $\rho(F_n)$ is Zariski connected. 
Since any non-solvable Lie group admits a semisimple quotient and by Lie's theorem a connected solvable complex Lie group can be conjugated to the group of upper triangular matrices, Conjecture \ref{conj:PP} is equivalent to:

\begin{conj}\label{conj:PP2}
If $\rho:F_n\to G$ is a representation with Zariski dense image, where $G\le \GL(d,\BC)$ is a connected simple algebraic group, then $\rho(\mathcal{P}_n)\not\subset {U}(d)$.
\end{conj}

The variety $\mathcal{U}(n,d)$ can be described explicitly as follows. Let $N\le \GL(d,\BC)$ be the group of upper triangular unipotent matrices. Then 
$$
 \mathcal{U}(n,d)=\{f^g:f\in\text{Hom}(F_n,N),g\in \GL(d,\BC)\},
$$
where $f^g(x):=gf(x)g^{-1}$. In particular, it follows that $\mathcal{U}(n,d)$
is irreducible. Thus we conjecture:

\begin{conj}\label{cong:irreducible}
The variety $\mathcal{PU}(n,d)$ is irreducible.
\end{conj}

\begin{thm}\label{thm}
Conjecture \ref{cong:irreducible} and Conjecture \ref{conj:PP} are equivalent.  
\end{thm}

The proof relies on the following: 

\begin{lem}[A dynamical criterion for detecting nontrivial eigenvalues]\label{lem:nonunipotent}
Let $g\in\SL_n(\BC)$ and suppose that there are $x\in \BP^{n-1}(\BC)$ and $0<r_2<r_1$, such that $g\cdot B(x,r_1)\subset B(x,r_2)$. Then $g$ is not unipotent.
\end{lem}

\begin{proof}
By taking iterations, it is easy to see that $g$ must have a (unique) fixed point in $B(x,r_1)$. Moreover, the Jacobian of $g$ at that fixed point is smaller than $1$. However, a direct computation shows that the differential of a unipotent matrix at a fixed point in the projective space is also unipotent.
\end{proof}

We will make use of the following consequence of Lemma \ref{lem:nonunipotent}. Recall that a unipotent is called {\it regular} if the second diagonal of its Jordan form consists of $1$'s only. Equivalently, a unipotent element is regular if its conjugacy class is dense in $U(d)$. We shall make use of the fact that a regular unipotent admits a unique invariant subspace of $\BC^d$ of any dimension $m=1,\ldots,d-1$. 

It is not hard to verify by a direct computation that iterated powers of a regular unipotent contract uniformly, away from the invariant hyperplane. Formally, if $u$ is a regular unipotent with fixed projective point $p$ and invariant projective hyperplane $H$ then for every $\epsilon$ there is $n_\gep$ (which depends also on the metric one choses for $\BP^{d-1}$) such that for $n\ge n_\gep$, $u^n$ contract the complement of the $\gep$-neighborhood of $H$ into the $\gep$-ball around $p$.  Thus the following is a consequence of Lemma \ref{lem:nonunipotent}.

\begin{cor}
Let $u$ be a regular unipotent with fixed projective point $p$ and invariant hyperplane $H$, and let $a$ be any linear transformation. If $a\cdot p\notin H$ then for all large $k$, $au^k$ is not a unipotent.
\end{cor}

Let us now prove Conjecture \ref{conj:PP2} under the assumption that some primitive basis is sent into the big set of regular unipotent.

\begin{thm}\label{thm:PP-regular}
Let $n\ge 3$ and let $\rho\in \mathcal{PU}(n,d)$. 
Suppose further that the $\rho(x_i),~i=1,\ldots,n$ are regular, where $(x_1,\ldots,x_n)$ is some free basis of $F_n$. Then $\rho\in\mathcal{U}(n,d)$.
\end{thm}


Denote $a_i=\rho(x_i)$. For simplicity let us first explain the proof of Theorem \ref{thm:PP-regular} under the assumption that $n\ge 4$.

\begin{clm}\label{clm:inv-sbsps}
For every $i,j$ the group $\langle a_i,a_j\rangle$ is not irreducible.
\end{clm}

We will show that $\langle a_2,a_3\rangle$ is not irreducible. Let $v_1,H_1$ be an invariant unit vector and the invariant hyperplane of $a_1$. Then for every $a\in \langle a_2,a_3\rangle$, $aa_4\cdot v_1\in H_1$ for otherwise the previous lemma implies that the element $aa_4a_1^k$ is not unipotent for large $k$. However $aa_4a_1^k$ is the image of the primitive element $xx_4x_1^k$ where $x$ is the corresponding element in $\langle x_2,x_3\rangle$, and as such it is unipotent by assumption.
Thus 
$$
 \text{span}\langle a_2,a_3\rangle\cdot (a_4 v_1)
$$ 
is an $\langle a_2,a_3\rangle$-invariant subspace of $H_1$.
\qed

\medskip

Note that the element $a_1$ played no special role in the proof above. Thus if $H_k$ denotes the invariant hyperplane of $a_k$ we deduce from the proof:
%
\begin{cor}
The $\langle a_i,a_j\rangle$ invariant subspace we found is contained in $H_k$ for every $k=1,\ldots,n$.
\end{cor}

We can now conclude by induction that $\langle a_1,\ldots,a_n\rangle$ is reducible. Indeed, as $a_1$ is regular, there are only finitely many options for the $\langle a_1,a_2\rangle$ invariant proper subspace. Substituting  $a_3^ka_2$ (so that the latter is still regular) instead of $a_2$ for more $k$'s than the dimension we derive that one of these finitely many subspaces is also $a_3$ invariant. Proceeding in the same way we eventually find a proper subspace which is invariant under all the generators. 

Hence the action is reducible. Say $V\le \BC^d$ is a proper invariant subspace. By applying the above argument to $V$ and to $\BC^d/V$, we may proceed and conclude that there is an invariant complete flag. This completes the proof of Theorem \ref{thm:PP-regular} for $n\ge 4$.

\medskip

Let us now explain the argument for $n\ge 3$. Arguing as in the proof of Claim \ref{clm:inv-sbsps},
with $\langle a_2\rangle$ playing the role of $\langle a_2,a_3\rangle$ and $a_3$ the role of $a_4$ we get that $a_2$ stabilizes some proper subspace of $H_1$. Since the indices $1,2$ are arbitrary, it follows that for each $i,j$, the element $a_i$ stabilizes some invariant subspace of $H_j$. Denoting by $v_i$ the invariant line of $a_i$, we deduce: 

\begin{cor}
For all $i,j$ 
$$ 
 v_i\in H_j.
$$
\end{cor}

We may suppose that our group is irreducible for otherwise we may proceed by induction on the dimension as above.
For each $i,k$ denote by $H_i^k$ the $k$-dimensional subspace invariant under $a_i$.
Let us say that a primitive tuple consisting of regular unipotent elements $(a_1,\ldots,a_n)$ is in a {\it good position} if for any $1\le i\ne j\le n$ and $k\le n$ we have $H_i^k\ne H_j^k$,
i.e. if any pair $(a_i,a_j)$ generates an irreducible subgroup.
Up to deforming the original tuple, by performing operations of the form $a_i\mapsto a_j^ka_i$, we may assume that it is in good position.

Let us assume that the $a_2$-invariant subspace of minimal dimension which contains $v_1$ is of maximal dimension among all possibilities (i.e. when allowing to apply Nielsen transformations as long as the new tuple is in good position and consists of regular elements). Suppose the dimension of this space is $k<d$. By maximality of $k$, it follows that $v_1\in H_j^{k}$ for every $j$.

We can find $k$ translations $a_2^{i_1}\cdot v_1,\ldots ,a_2^{i_k}\cdot v_1$ that span $H_2^k$ so that if we replace $a_1$ by $a_2^{i_j}a_1a_2^{-i_j}$, each of the $k$ new tuples 
$$
 (a_2^{i_j}a_1a_2^{-i_j},a_2,\ldots,a_n),~j=1,\ldots,k
$$ 
will also be in good position. However, $a_2^{i_j}\cdot v_1$ is the invariant vector of $a_2^{i_j}a_1a_2^{-i_j}$, so again, by maximality of $k$ we have:
$$
 a_2^{i_j}\cdot v_1\in H_3^k.
$$
Thus $H_3^k=H_2^k$ in contrary to the assumption of good position.

This completes the proof of Theorem \ref{thm:PP-regular}.
\qed


\begin{proof}[Proof of Theorem \ref{thm}]
Fix a primitive basis $(x_1,\ldots,x_n)$ for $F_n$.
It follows from Theorem \ref{thm:PP-regular} that if for $\rho:F_n\to\GL_d(\BC)$ the elements $\rho(x_i),~i=1,\ldots,n$ are regular unipotent then $\rho \notin \mathcal{PU}\triangle\mathcal{U}$, i.e. if $\rho\notin\mathcal{U}$ it is also $\notin\mathcal{PU}$. Since the set
$$
 X=\{\rho:\rho(x_1),\ldots,\rho(x_n)~\text{are regular unipotent}\}
$$  
clearly intersects $\mathcal{PU}$ in a Zariski open set, we derive that if every irreducible component of $\mathcal{PU}$ intersects $X$ nontrivially then $\mathcal{U}=\mathcal{PU}$. In particular this is the case if $\mathcal{PU}$ is irreducible. 
\end{proof}


\section{Baby Wiegold's conjecture}\label{sec:BB}

A. Lubotzky promoted the following conjecture, under the name `Baby Wiegold conjecture':

\begin{conj}[Lubotzky]\label{conj:Baby-Wiegold} 
For $n\ge 3$, the free group $F_n$ does not contain a finite index characteristic subgroup $\gD\lhd F_n$ such that $G=F_n/\gD$ is simple. 
\end{conj}

Equivalently, for any $n\ge 3$ and a finite simple group $G$, there is no $\aut(F_n)$-fixed point in $\text{Epi}(F_n,G)/\text{Aut}(G)$.
Obviously, this is a weak version of the Wiegod conjecture which says that the action of $\aut(F_n)$ on $\text{Epi}(F_n,G)/\text{Aut}(G)$ has a unique orbit. 

Let us consider the analog scenario in the algebraic group setup.

\begin{defn}
A linear representation $\rho:F_n\to \SL_d(\BC)$ is called {\it characteristic} if $\rho\circ\gs$ is conjugated to $\rho$ in $\SL_d(\BC)$ for every $\gs\in \aut(F_n)$, that is if the corresponding character is an $\aut(F_n)$-fixed point in the character variety.
\end{defn}

The following is an analog of Conjecture \ref{conj:Baby-Wiegold} in the algebraic group context:

\begin{conj}\label{conj:fixed}
Suppose that $n\ge 3$, and let $\rho$ be a characteristic linear representation of 
$F_n$. Then $\rho(F_n)$ is virtually solvable.
\end{conj}

Another variant is:

\begin{conj}
If $\overline{\rho(F_n)}^Z$ is simple then $\rho$ is not characteristic.
\end{conj}

If $\rho$ is characteristic then $\rho(\mathcal{P}_n)$ is contained in a single conjugacy class of $\SL_d(\BC)$. 
We shall regard that conjugacy class as the {\it type} of $\rho$.
 
Observe that if we identify $F_n$ with the normal subgroup 
$$
 F_n\cong\text{Inn}(F_n)\lhd\aut(F_n)
$$ 
then the restriction $\rho |_{F_n}$ of any linear representation 
$\rho:\aut(F_n)\to\GL_d(\BC)$ is characteristic. Thus
Conjecture \ref{conj:FZL} of Formanek--Zelmanov and Lubotzky is a special case of the following special case of Conjecture \ref{conj:fixed}: 

\begin{conj}[Generalized FZL]\label{conj:G-FZL}
Let $n\ge 3$ and let $\rho:F_n\to\SL_d(\BC)$ be a characteristic representation whose type is virtually unipotent. Then $\rho(F_n)$ is virtually solvable.  
\end{conj}

In view of Theorem \ref{thm:PP-regular}, Conjecture \ref{conj:G-FZL} holds for characteristic representations whose type is regular unipotent. 
Furthermore, for such representations the image $\rho(F_n)$ is a unipotent group and in particular nilpotent.


\section{Cosets in conjugacy classes}
The group $\aut({F_n})$ is generated by Nielsen transformations, and the primitive elements of $F_n$ form a conjugacy class in $\aut(F_n)$. 
This is one of the motivations to study representations $\rho$ of $F_n$ where $\rho(\mathcal{P}_n)$ is contained in a conjugacy class. 
If $(x_1,\ldots,x_n)$ is a primitive tuple in $F_n$ then any element of the form $x_1w(x_2,\ldots,x_n)$ is primitive, where $w$ is a word in $n-1$ variables. That is $x_1\langle x_2,\ldots,x_n\rangle\subset\mathcal{P}_n$. Thus, given an algebraic group $G$ we are particularly interested in the following:

\begin{prob}\label{ques:conj-coset} 
Classify pairs $(H,x)$ of a subgroup $H\le G$ and an element $x\in G$, such that $xH$ is contained in $x^G=\{gxg^{-1}:g\in G\}$. 
\end{prob}

Note that $xH\subset x^G \iff Hx\subset x^G \iff HxH\subset x^G$.

\begin{exam} For $G=GL_2$ the only examples (up to conjugation and replacing $H$ with its subgroup) are:
\begin{enumerate}
\item $x$ is diagonal, $H$ upper unitriangular.
\item $x$ is monomial, $H$ is diagonal.
\end{enumerate}

\end{exam}

However, already in $\GL_3$ the picture is much more complicated as we shall demonstrate below.
Yet, we believe that the following always holds:

\begin{conj}\label{conj:coset}
Let $H\le G=\GL_d(\BC)$ be a group such that for some $x\in G$ we have $Hx\subset x^G$. Then $H$ is virtually solvable.
\end{conj}

\begin{rem}
Note that the condition $Hx\subset x^G$ is equivalent to 
$$
 H\subset [G,x]:=\{ gxg^{-1}x^{-1}:g\in G\}.
$$
\end{rem}

Let us give two additional formulations for Conjecture \ref{conj:coset}. The first may suggest 
to try to apply dynamics on projective spaces, while the second may suggest
that the problem could be approached by representation theoretic tools.

\begin{conj}\label{conj:9.6}
There is no faithful representation of the free group $\rho:F_2\to \GL_d(\BC)$ whose image satisfies 
$$
 \rho(F_2)\subset [\GL_d(\BC),g]
$$
for some $g\in \GL_d(\BC)$.
\end{conj}

The equivalence between Conjectures \ref{conj:9.6}
and Conjecture \ref{conj:coset} is due to the Tits alternative, namely,  that a linear group is not virtually solvable iff it contains a free subgroup.

A particularly interesting case of Conjecture \ref{conj:coset} is when restricting to the group $H=\SL_2(\BC)$ for which the representation theory is well understood:

\begin{conj}\label{conj:9.5}
There is no nontrivial representation of $\rho:\SL_2(\BC)\to\SL_d(\BC)$ whose image satisfies 
$$
 \rho(\SL_2(\BC))\subset [\SL_d(\BC),g]
$$
for some $g\in \SL_d(\BC)$.
\end{conj}

Recall that a complex linear algebraic group is not virtually solvable iff it contains $\SL_2(\BC)$ as a subgroup. 
Note also that the Zariski closure of $[G,g]$ coincides with the Hausdorff closure. Thus the following Conjecture implies Conjecture \ref{conj:coset}:

\begin{conj}
There is no nontrivial representation of $\SL_2(\BC)$ whose image satisfies 
$$
 \rho(\SL_2(\BC))\subset \overline{[\SL_d(\BC),g]}
$$
for some $g\in \SL_d(\BC)$.
\end{conj}

Since $[G,g]$ is open in its closure, we can differentiate and obtain a formulation in terms of Lie algebras: 

\begin{conj}
There is no faithful linear representation of the Lie algebra $\text{sl}_2(\BC)$ whose image is contained in the range of the linear map $\text{Ad}(g)-1$
for some $g\in \SL_d(\BC)$.
\end{conj}

%
%
%
%
%

\subsection{One parameter cosets}

As a first step towards Problem \ref{ques:conj-coset} let us consider cosets of one parameter subgroups in $G=\SL_d(\BC)$. The following quite general result indicates that the examples of such cosets which consist of conjugated elements are very particular. 


\begin{prop} 
Suppose $H=\{e^{tA}:t\in \BR\}$ for some $A\in\text{sl}_d(\BC)$. Let the eigenvalues of $A$ be $\gl _1 ,\ldots , \gl _n$. Assume that every partial sum (except the whole sum) of the $\gl _i$'s is non zero. Suppose that $xH\subset x^G$. Then $x$ is conjugated to $z=\diag(1,\zeta,\zeta^2,\ldots,\zeta^{n-1})$ where $\zeta=1^\frac{1}{n}$ a primitive $n$'th root of unity.
\end{prop}

\begin{proof} The trace $tr(xe^{tA})$ is constant. If we differentiate $k$ times and substitute $t=0$, we get $tr(xA^k)=0$. The matrix $A$ is nonsingular, so it satisfies a polynomial with nonzero constant term. Hence $tr(x)=0$. The same argument applied to wedge powers gives that $tr(\bigwedge ^m x)=0$, and hence the characteristic polynomial of $x$ is $t^n-c$ with $c=det(x)$. Moreover, supposing $x\in\SL_n(\BC)$, we deduce that $x\sim z$. 
\end{proof}

Another case of a particular interest is when $x$ is unipotent.
We are particularly interested in understanding cosets of subgroups which consist of unipotent elements. Here is one simple observation:

\begin{prop}\label{exam:eigenvalue1}
Let $H\le\SL_d(\BC)$ be a group such that for some $x\in\SL_d(\BC)$ we have $xH\subset {U}(d)$.
Then every element of $H$ has $1$ as an eigenvalue.
\end{prop}

Indeed, for $A\in H$ look at the characteristic polynomial $P_A(t)=\sum_{i=0}^d a_i t^i$. Then
$$
 0=tr (xP_A(A))=\sum_{i=0}^d a_i\text{tr} (xA^i)=\sum_{i=0}^d a_i d \Rightarrow P_A(1)=\sum_{i=0}^d a_i=0.
$$

Such a result is useful. For instance one may deduce that
${U}(4)$ contains no coset of a non-virtually-solvable subgroup $H$.
To see this, we may argue by contradiction supposing, 
without loss of generality, that $H$ is isomorphic to $\SL_2(\BC)$. Since $H$ is simple the space $\BC^4$ decomposes to a direct sum of $H$-irreducible subspaces. By Proposition \ref{exam:eigenvalue1} every element of $H$ has an eigenvalue $1$. This excludes the case where $\BC^4=\BC^2\oplus\BC^2$ with $H=\SL_2$ acts via the standard representation on each factor. In all other cases it is not hard to show that for every nontrivial unipotent $u$ there is $g\in H$ such that either $ug$ or $u^{-1}g$ is proximal, hence satisfies the condition of Lemma \ref{lem:nonunipotent}, and therefore is non-unipotent (see \cite{dense} for the definition and some analysis of proximal elements).

However the assumption that $uH\subset \mathcal{U}(d)$ does not imply that $H$ is unipotent. The following example was shown to me by M. Larsen.

\begin{exam}[A coset of a torus consisting of unipotents] 
Let 
$$ u=\bmat0&1&3\\0&3&8\\-1&0&0\emat, h=\bmat t&0&0\\0&1&0\\0&0&t^{-1}\emat. 
$$
Then one can compute that 
$$
 det(uh)=1~\text{and}~tr(uh)=tr((uh)^2)=3.
$$ 
It follows that $uh$ is unipotent for every $t\ne 0$. 
\end{exam}

Inspired by Larsen's example we can construct similar semisimple examples:

\begin{exam}
There are two conjugate semisimple elements $a,b$ such that $ab^n$ is conjugate to $a$ for every $n$: Take 
$$
 a=\bmat0&-1&2\\0&-4&\frac{15}{2}\\2&0&0\emat~\text{and}~b=\bmat t&&\\&1&\\&&t^{-1}\emat.
$$ 
Then $det(ab),tr(ab),tr((ab)^2)$ are independent of $t$, so $ab^n$ are all conjugate. Moreover, $tr(a)=-4,det(a)=1$, and $a$ has an eigenvalue 1. Therefore it is conjugate to $b$ for some value of $t$.
\end{exam}


\end{document}